\documentclass{amsart}

\usepackage{amsthm}
\usepackage{amsmath}
\usepackage{amssymb}
\usepackage{amscd}
\usepackage{mathrsfs}
\usepackage{verbatim}
\usepackage{hyperref}
\usepackage{xcolor}
\input{colordvi}

\allowdisplaybreaks

\newcommand{\PP}{\mathbb{P}}

\newcommand{\sign}{\mbox{sign}}

\newcommand{\DEQSZ}{\begin{eqnarray}}
\newcommand{\EEQSZ}{\end{eqnarray}}

\newcommand{\eps}{\varepsilon}

\newcommand{\CL}{{\mathcal{L}}}
\newcommand{\CF}{{\mathcal{F}}}

\newcommand{\DEQS}{\begin{eqnarray*}}
\newcommand{\EEQS}{\end{eqnarray*}}
\newcommand{\lk}{\left}
\newcommand{\rk}{\right}
\newcommand\del[1]{}
\newcommand\think[1]{}
\newcommand\new[1]{}
\newcommand\zus[1]{}

\def\R{{\mathbb R}}
\def\N{{\mathbb N}}
\def\C{{\mathbb C}}
\def\E{{\mathbb E}}
\def\P{{\mathbb P}}

\newcommand{\n}{\Vert}

\newcommand{\calF}{\mathcal F}
\newcommand{\calH}{\mathcal H}
\newcommand{\calL}{\mathcal L}

\newcommand{\MA}{{\bf (A)}}
\newcommand{\MF}{{\bf (F)}}
\newcommand{\MG}{{\bf (G)}}
\newcommand{\MFloc}{{\bf (F$_\textrm{loc}$)}}
\newcommand{\MGloc}{{\bf (G$_\textrm{loc}$)}}

\newcommand{\pars}{\par}

\theoremstyle{plain}
\newtheorem{theorem}{Theorem}[section]
\theoremstyle{remark}
\newtheorem{remark}[theorem]{Remark}
\newtheorem{example}[theorem]{Example}

\theoremstyle{plain}
\newtheorem{corollary}[theorem]{Corollary}

\newtheorem{proposition}[theorem]{Proposition}
\newtheorem{definition}[theorem]{Definition}

\newtheorem*{notation}{Notation}

\newcommand{\inv}[1]{\frac{1}{#1}}
\newcommand{\tinv}[1]{\tfrac{1}{#1}}

\newcommand{\minsym}{\wedge}

\newcounter{gr1}

\newcommand{\sgc}{\color{red}}
\newcommand{\cgs}{\color{black}}

\begin{document}
\title[Pathwise approximations to parabolic SPDEs]{Pathwise space approximations of semi-linear parabolic SPDEs with
multiplicative noise}

\author{Sonja Cox \and Erika Hausenblas}
\thanks{Sonja Cox: Universit\"at Innsbruck/Delft University of Technology. Email address: \url{sonja.cox@uibk.ac.at}.}
\thanks{Erika Hausenblas  (corresponding author): Montana Universit\"at Leoben. Email address: \url{erika.hausenblas@unileoben.ac.at}}

\begin{abstract}
We provide convergence rates for space approximations of semi-linear stochastic differential
equations with multiplicative noise in a Hilbert space.  The space approximations we consider are spectral Galerkin and finite elements, and the type of convergence we consider is almost sure uniform convergence, i.e., pathwise convergence. The proofs are based on a recent perturbation result for such
equations. 
\vspace{0.3cm}
\end{abstract}

\maketitle

\section{Introduction}\label{mnb}

Recently, the authors obtained a perturbation result for
stochastic differential equations in the class of \textsc{UMD}
Banach spaces \cite{coxHau:11}. This class of spaces includes the
Hilbert spaces, and in this article we focus on the Hilbert-space
setting only. We shall illustrate how this abstract perturbation
result of \cite{coxHau:11} can be used to prove so-called
\emph{pathwise convergence} of Galerkin and finite element
approximations for stochastic differential equations in Hilbert
spaces. By combining these results with e.g.\ the time
discretization results in \cite{CoxNeer:11} or \cite{Pri:01} one
can obtain pathwise convergence of a fully discretized
scheme\footnote{\sgc{}The proof of Proposition 4.2 in the old (published) version of this article contained a mistake in the final estimate. This mistake has been repaired here (see Proposition~\ref{prop:finEl} below), all changes with respect to the published article are marked red. The mistake did not affect the main results. The authors thank Kristin Kirchner for pointing out the mistake.\cgs{}}.\pars

Recall that a stochastic differential equation in a Hilbert space arises when taking the functional-analytic approach to a stochastic partial differential equation (SPDE),
 see \cite{DaPZab:92}. In this article we consider equations of the following type:\par

\begin{equation}\label{SDE-intro}\tag{SDE}
\left\{ \begin{aligned} dU(t) & = AU(t)\,dt + F(t,U(t))\,dt + G(t,U(t))\,dW_H(t);\quad t\in [0,T],\\
U(0)&=x_0. \end{aligned}\right.
\end{equation}
Here $A$ is an unbounded operator generating an analytic $C_0$--semigroup $(S(t))_{t\geq 0}$ on a Hilbert space $\calH$. Furthermore, $H$ is another Hilbert space and $W_H$ is an
$H$-cylindrical Brownian motion on a probability space $(\Omega,(\CF_t)_{t\ge0},\PP)$. The non-linearities $F:[0,T]\times \calH\to \calH^{A}_{\theta_F}$, $\theta_F>-1$, and
$G:[0,T]\times \calH\to \CL_2(H,\calH^{A}_{\theta_G})$, $\theta_G>-\inv{2}$, are assumed to satisfy global Lipschitz and linear growth conditions in $\calH$, uniformly on $[0,T]$
 (although these assumptions can be weakened, as we will explain later). Note that $\calH^{A}_{\delta}$ denotes the fractional domain space $D((\lambda-A)^{\delta})$ for $\delta>0$
 and the extrapolation space $\overline{\calH}^{\n (\lambda-A)^{\delta}\cdot\n}$ for $\delta<0$, and that $\calL_2(H,\calH^{A}_{\theta})$ denotes the space of Hilbert-Schmidt operators
  from $H$ into $\calH^{A}_{\theta}$. For the precise assumptions on $F$ and $G$ see Section \ref{ss:setting}.\par

\begin{example}\label{example}
A typical example of a problem that fits in the framework above is the one-dimensional parabolic equation with space-time
white noise:
\begin{equation*}% \label{SDE-intro-example}
\left\{ \begin{aligned} du(t,\xi ) & = \frac d {d\xi } \lk( a(\xi ) \frac d {d\xi } \, u(t,\xi )\rk) \, dt +
f(t,u(t,\xi ))\,dt + g(t,u(t,\xi ))\,dW(t,\xi );\\ & \hspace{6cm}\, t\in [0,T], \xi \in[0,1];\\
u(t,0)&=u(t,1)=0;\\
u(0,\xi )&=x_0(\xi ),\quad \xi \in[0,1]. \end{aligned}\right.
\end{equation*}
Here we assume that $a\in L^{\infty}(0,1)$ is bounded away from zero, and that
$f:[0,T]\times \R\rightarrow \R$ and $g:[0,T]\times\R\rightarrow \R$ satisfy certain Lipschitz
conditions. This fits into the setting of \eqref{SDE-intro} if we take $H=\calH=L^2(0,1)$, $F(t,U(t))(x)=f(t,u(t,x))$, $(G(t,U(t))h)(x)=g(t,u(t,x))h(x)$ for $h\in L^2(0,1)$, $\theta_F=0$ and
$\theta_G=-\inv{4}-\eps$ (see
\cite{Pri:01}). For more examples we refer to \cite[Section
10]{NVW08} and the introduction
of \cite{Pri:01}. \par % See Engel and Nagel, chapter on elliptic eqns.
\end{example}

In the subsections below we briefly elaborate on the approximation methods we consider, explaining our main results, and on the concept of pathwise convergence.

\subsection*{The spectral Galerkin method}
Suppose the spectrum of $A$ consists
only of eigenvalues $(\lambda_n)_{n\in\N} \subset (-\infty,\omega]$ for some $\omega\in \R$. Assume $(\lambda_n)_{n\in\N}$ is ordered such that $\lambda_{n+1}\leq
\lambda_{n}$ for all $n\in\N$ and let $\phi_n$ denote the eigenvector corresponding to $\lambda_n$.
Set $\calH_n :=\textrm{span}\{\phi_1,\ldots,\phi_n\}$ and let $P_n:\calH\rightarrow \calH_n$ be the orthogonal projection of $ \calH$ onto $\calH_n$. The $n^{\textrm{th}}$ Galerkin approximation $U^{(n)}(t)=\sum_{k=1}^{n}u_k(t)\phi_k$, $u_k\in L^p(\Omega,C([0,T]))$, $p\in (2,\infty)$, is obtained by solving, for $k=1,\ldots,n$,
\begin{equation}\label{galerkin_n}
\begin{aligned} u_k(t) & = \langle x_0,\phi_k\rangle_{\calH} + \lambda_k \int_{0}^{t} u_k(s) \,ds + \int_{0}^{t} \Big\langle F\Big(s, \sum_{k=1}^{n}u_k(s)\phi_k\Big), \phi_k \Big\rangle_{\calH} \,ds \\
& \quad + \int_{0}^{t} G^*\Big(s, \sum_{j=1}^{n}u_j(s)\phi_j\Big)\phi_k \,dW_H(s); \qquad t\in[0,T].
\end{aligned}
\end{equation}
The adjoint of $G(s, \sum_{j=1}^{n}u_j(s)\phi_j)\in \calL(H,\calH)$ is denoted by $G^*(s, \sum_{j=1}^{n}u_j(s)\phi_j)\in \calL(\calH,H)$.\pars

Note that the stochastic integral in the equation above still involves a (possibly) infinite dimensional Brownian motion. Whether this resolves into a integral with respect to a finite-dimensional Brownian motion depends on the choice of $G$ and on the representation of the noise. We shall briefly discuss two examples.\pars

If the noise is additive, i.e., if $G\equiv g\in \calL_2(H,\calH_{\theta_G}^{A})$, then there exists a sequence $(h_j)_{j = 1}^{\infty}\subset H$, the non-zero terms of which form an orthonormal basis for $H$, and $(g_j)_{j= 1}^{\infty}\subset \R$ such that $g$ may be represented as follows: $$ g= \sum_{j= 0}^{\infty} g_j \phi_j \otimes h_j. $$
(Note  that $g \in \calL_2(H,\calH_{\theta_G}^{A})$ if and only if $(\lambda_j^{\theta_G} g_j \n h_j\n )_{j=1}^{\infty} \in \ell_2$.)
Representing $W_H$ by setting $W_H=\sum_{j=1}^{\infty}W_j h_j$, with $(W_j)_{j=1}^{\infty}$ independent standard $\R$-valued Brownian motions, we have, for $j=1,\ldots,n$,
$$  \int_{0}^{t} G^*\Big(s, \sum_{j=1}^{n}u_j(s)\phi_j\Big)\phi_k \,dW_H(s) = g_k W_k(t).$$\pars

As for an example with multiplicative noise, consider Example
\ref{example} with $a\equiv 1$. In that case $\phi_k(x)=\sin k\pi
x$, $k=1,\ldots,\infty$. Suppose the non-linear term $g$ is given
by $g(t,u(t,x))=u(t,x)$. Note that $(\sqrt{2}\cos \ell\pi
x)_{\ell=1}^{\infty}$ is an orthonormal basis for $L^2=L^2(0,1)$.
Let $(W_\ell)_{\ell=1}^{\infty}$ be independent standard
$\R$-valued Brownian motions and set $W_{L^2} :=
\sqrt{2}\sum_{\ell=1}^{\infty}W_{\ell} \cos \ell \pi x$. We then have,
for $j=1,\ldots,n$,
\begin{align*}
&   \int_{0}^{t} G^*\Big(s, \sum_{j=1}^{n}u_j(s)\phi_j\Big)\phi_k \,dW_H(s) \\
& \qquad \quad = \sqrt{2}\sum_{\ell=1}^{\infty}\sum_{j=1}^{n} \int_{0}^{t}  u_j(s) \int_{0}^{1} \sin j\pi x \sin k\pi x \cos \ell \pi x \,dx \,dW_{\ell}(s)\\
& \qquad \quad
=  2^{-\frac{3}{2}}\Big(\sum_{\ell=1, \ell\neq k}^{n+k}  \sign (k-\ell) \int_{0}^t  u_{|k-\ell|}(s) \,dW_{\ell}(s) + \sum_{\ell=1}^{n-k}\int_{0}^{t} u_{k+\ell}(s) \,dW_{\ell}(s)\Big).
\end{align*}
The right-hand side above involves $2n-1$ terms, which means that the system obtained from the equations for $\phi_j, j=1,\ldots,n,$ involves $n(2n-1)$ stochastic integrals. However, one may check that taking the representation $W_{L^2} := 2\sum_{\ell=1}^{\infty}W_\ell \sin \ell\pi x$ leads to infinitely many stochastic integrals.

In Section \ref{s:galerkin} we prove the following to be a direct consequence of Proposition \ref{prop:galerkin} below:
\begin{theorem}\label{c:galerkin_pathwise}
Suppose there exists an $\alpha>0$ and a constant $C$ such that for all $n\in \N$ we have
\begin{align*}
| \lambda_n | \geq C n^{\alpha}.
\end{align*}
Let $\eta\in [0,1]$ and $p\in (2,\infty)$ be such that
$$\eta+\tinv{\alpha p}<\min\{1+\theta_F,\tfrac{1}{2}+\theta_G-\tinv{p} \}$$ and assume $x_0\in L^{p}(\Omega,\mathcal{F}_0,\calH_{\eta}^A)$.
Then there exists a random variable $\chi_{\eta}\in L^p(\Omega)$ such that
\begin{align*}
\n U - U^{(n)}\n_{C([0,T],\calH)}
&\leq \chi_{\eta} n^{-\alpha\eta}.
\end{align*}
\end{theorem}
In particular, for Example \ref{example} with $a\equiv 1$ we have $\lambda_n = \pi^2 n^2$ and thus the convergence rate is $n^{-\inv{2}+\eps_0}$ for $\eps_0$ arbitrarily small, both in $L^p(\Omega,C([0,T],L^2))$ and almost surely, provided $x_0\in L^{p}(\Omega,\mathcal{F}_0,W^{1,2}_0)$ for $p$ sufficiently large. Here $W^{1,2}_0=W^{1,2}_0(0,1)$ denotes space of functions $f$ in the Sobolev space $W^{1,2}$ which satisfy $f(0)=f(1)=0$.\pars

\subsection*{The finite element method}
In Section \ref{s:finEl} we prove convergence of a finite element method for a class of elliptic second-order differential equations. However, in order to obtain convergence we shall need that $G$ maps into $\calL_2(H,\calH_{\theta_G})$ for $\theta_G>0$. This generally translates to certain smoothness assumptions on the noise. In particular, no rates are obtained for Example \ref{example}.\pars

The type of equations we consider includes the case that $A$ is a self-adjoint second order differential operator on $\calH=L^2(\begin{color}{red}D\end{color})$, $D\subset \R^d$ a convex polyhedron, with \begin{color}{red}homogeneous Dirichlet\end{color} boundary conditions. For the precise assumptions on $A$ we refer to Section \ref{s:finEl}.\pars

Let $a:W^{1,2}(\begin{color}{red}D\end{color})\times W^{1,2}(\begin{color}{red}D\end{color})\rightarrow \R$ denote the form
associated with $A$ (i.e.\ $a(u,v)=\langle Au,v\rangle_{L^2(\begin{color}{red}D\end{color})}$ for all $u,v\in D(A)$). Let $V_{n}\subset W^{1,2}(\begin{color}{red}D\end{color}), n\in \N$, be a family of finite-element spaces for which that the maximal diameter of the support of the elements of $V_n$ is $n^{-1}$ \begin{color}{red}(the exact assumptions on the finite element spaces are presented in Section~\ref{s:finEl})\end{color}. Set $X_n:=(\{f\in V_n\},\n \cdot \n_{L^2(\begin{color}{red}D\end{color})})$. The finite-element approximation $U^{(n)}$ to the solution to \eqref{SDE} with $A$ as above and initial condition $u_0$, is the element of $L^p(\Omega,C([0,T],X_n))$ satisfying, for all $t\in[0,T]$,
\begin{equation}\label{finelsol_intro}
\begin{aligned}
\langle U^{(n)}(t) , v_n \rangle_{L^2(\begin{color}{red}D\end{color})} & = \int_{0}^{t} a(U^{(n)}(s),v_n)\,ds + \int_{0}^{t} \langle
F(s,U^{(n)}(s)), v_n\rangle_{L^2(\begin{color}{red}D\end{color})} \,ds\\
&\quad + \int_{0}^{t} G^*(s,U^{(n)}(s)) v_n \,dW_H(s), \qquad \textrm{a.s.\ for all } v_n\in V_n,\\
\langle U^{(n)}(0) , v_n \rangle_{L^2(\begin{color}{red}D\end{color})} & = \langle u_0, v_n \rangle_{L^2(\begin{color}{red}D\end{color})}, \qquad \textrm{a.s.\ for all } v_n\in V_n.
\end{aligned}
\end{equation}
Here $G^*(s,U^{(n)}(s)) \in \calL(L^2(\begin{color}{red}D\end{color}),H)$ is the adjoint of $G(s,U^{(n)}(s)) \in \calL(H,L^2(\begin{color}{red}D\end{color}))$.

Note that once again the stochastic integral in the equation above involves a (possibly) infinite dimensional Brownian motion. As before, it depends very much on the choice of $(v_n)_{n\in \N}$, $G$, and the representation of $W_H$, whether the stochastic integral resolves into a integral with respect to finite-dimensional noise.\pars
Assuming that the elements of $V_n$ are given by (piecewise) first-order polynomials, we obtain the following convergence result as a direct consequence of Proposition \ref{prop:finEl} in Section \ref{s:finEl} below:
\begin{theorem}\label{c:finEl_pathwise}
Let $\eta\in [\inv{2},1]$ and $p\in (2,\infty)$ be such that
$$\tinv{2} \leq \eta+\tinv{p}<\min\{1+\theta_F,\tfrac{1}{2}+\theta_G-\tinv{p}\}$$ and assume $u_0\in L^{p}(\Omega,\mathcal{F}_0,W^{2\eta,2}(\begin{color}{red}D\end{color}))$, where $W^{2\eta,2}(\begin{color}{red}D\end{color})$ denotes the fractional Sobolev space.
Then there exists a random variable $\chi_{\eta}\in L^p(\Omega)$ such
that for all $n\geq 1$ we have
\begin{align*}
\n U - U^{(n)}\n_{C([0,T],L^2(\begin{color}{red}D\end{color}))}
&\leq \chi_{\eta}(\omega) n^{-2\eta}.
\end{align*}
\end{theorem}
Note that in order to obtain convergence from the theorem above it is necessary that $\theta_G > \inv{p}$, i.e., as mentioned above, $\theta_G$ must be sufficiently large. The condition $\eta\geq \inv{2}$ is due to the fact that the Riesz operator is not $L^2$-stable, see Remark \ref{rem:not_all_eta}.

\subsection*{Pathwise convergence}
We refer to the convergence of $U^{(n)}$ against $U$ as provided in Theorems \ref{c:galerkin_pathwise} and \ref{c:finEl_pathwise} as \emph{pathwise convergence}, because it implies that one has convergence of the approximation \emph{process} $U^{(n)}(\omega)$ for almost every fixed $\omega$ in the probability space $\Omega$.
 This as opposed to \emph{convergence in moments}, which concerns estimates of the following type:
\begin{align*}
\big(\E \n U - U^{(n)}\n_{C([0,T];\calH)}^{p}\big)^{\inv{p}}
&\leq C n^{-\alpha\eta},
\end{align*}
for some constant $C$ fixed. Pathwise convergence and convergence in moments are closely related: below we obtain pathwise convergence from convergence in moments by a Borel-Cantelli argument, and in the setting of Theorems \ref{c:galerkin_pathwise} and \ref{c:finEl_pathwise} the reverse follows immediately as the random variable $\chi$ has finite $p^{\textrm{th}}$ moments.\pars

It should be noted that many results in the literature deal with a type of convergence which we shall refer to as \emph{pointwise convergence}, being convergence of the following type:
\begin{align*}
\sup_{t\in [0,T]}\big(\E \n U(t) - U^{(n)}(t)\n_{\calH}^{p}\big)^{\inv{p}}
&\leq C n^{-\alpha\eta},
\end{align*}
for some constant $C$ fixed. This type of convergence is weaker than convergence in moments, yet for spectral Galerkin the same convergence rates are obtained -- see the discussion on `related work' below.

The advantage of pathwise convergence is that it allows us to weaken the assumptions on the non-linear terms $F$ and $G$. To be precise, we may assume that $F$ and $G$ are \emph{locally} Lipschitz (and of linear growth). This is demonstrated in Section \ref{s:local}, where we extend the Theorems \ref{c:galerkin_pathwise} and \ref{c:finEl_pathwise} to the case that $F$ and $G$ are locally Lipschitz and $x_0\in L^{0}(\Omega,\mathcal{F}_0,\calH_{\eta}^A)$.\pars
This advantage of pathwise convergence has already been investigated in \cite{Jen:09} for equations with additive noise, and applied again in \cite{CoxNeer:11} for time discretizations of equations with multiplicative noise. The approach taken in \cite{CoxNeer:11} and \cite{Jen:09} translates directly to the results obtained in this article. \par

\subsection*{Related work}
Results concerning pointwise convergence of the Galerkin method for stochastic differential equations with multiplicative noise have been obtained in \cite{haus:03} and extended to a
setting comparable to the setting we study in \cite{Yan:05}. This article seems to contain the first result concerning pathwise
convergence of Galerkin space approximations for multiplicative noise. In fact, to our knowledge the only result
concerning pathwise convergence of space approximations with multiplicative noise is given in \cite{Gyo:98}, where the
author considers convergence of the finite difference method for the one-dimensional heat equation with space-time white
noise. To be precise, the author considers the equation in Example \ref{example} with $a\equiv 1$ and $f$ and $g$ not only
time-dependent but also space-dependent. The author obtains pathwise convergence in probability, but
without convergence rates.\par

For additive noise, various pathwise convergence results have been obtained. For example, in \cite{KloeLordNeu:11} the authors consider Galerkin approximations in the same setting as we do but with additive noise and slightly stronger conditions on $F$. They obtain the same
convergence rates as we do. A similar pathwise result is obtained in aforementioned article \cite{Jen:09}. In that
article it is demonstrated how pathwise convergence for problems with globally Lipschitz coefficients can be used in
combination with localization to obtain convergence results for the case that $F$ and $G$ are merely locally Lipschitz.
This work is extended to higher order convergence rates in \cite{Jen:11} for the same problem with slightly more general types of noise.\par

All of the above mentioned articles concern fully discretized schemes. Recall that in order to obtain a fully discretized scheme from the results presented here \emph{and} maintain pathwise convergence, one would have to combine our results with pathwise convergence results for a time discretization scheme. We mentioned already that such results may be found in \cite{CoxNeer:11} or \cite{Pri:01}. We refer to \cite{JenKloe:09b} for a recent overview of such results.

\subsection*{Outline}
In Section \ref{s:prelim} we present the abstract perturbation
result on which our approximation results are based (Theorem \ref{t:app}). In order to present this result, we first introduce the preliminaries which are necessary to formulate this result. The perturbation result is applied in the Sections \ref{s:galerkin} and \ref{s:finEl} to prove Theorems \ref{c:galerkin_pathwise} and \ref{c:finEl_pathwise}. Finally, in Section \ref{s:local} we demonstrate how the pathwise convergence result can be used to weaken the assumptions on $F$ and $G$.
\begin{notation}
Let $D\subset \R^{d}$, $d=1,2,\ldots$, be open and bounded. For $s\in \R$ the (fractional) Sobolev space over $D$ is denoted by $W^{s,2}(D)$, where $s$ denotes the number of fractional derivatives that may be taken. Moreover, $W^{s,2}:= W^{s,2}(0,1)$ and $L^2:=L^2(0,1)$. Generally, $(\Omega,\calF,\P)$ denotes a probability space and $(\calF_t)_{t\geq 0}$ a filtration on this probability space.\pars

For an operator $A$ on a complex Hilbert space $\calH$ we denote by
$\rho(A)$ the resolvent set, i.e.\ all the complex numbers
$\lambda\in \mathbb{C}$ for which $\lambda I-A$ is (boundedly) invertible. For
$\lambda\in\rho(A)$ we denote by $R(\lambda:A)$ the resolvent of
$A$, i.e.\ $R(\lambda:A)=(\lambda I -A)^{-1}$. The spectrum of
$A$, i.e.\ the complement of $\rho(A)$ in $\C$, is denoted by $\sigma(A)$. If $\calH$ is a real Hilbert space, then $\rho(A), \sigma(A)$ and $R(\lambda:A)$ are defined over the complexification of $\calH$.\par
For $\calH_1, \calH_2$ Hilbert spaces we let $\calL(\calH_1,\calH_2)$ be the Banach space of all bounded linear operators from $\calH_1$ to
$\calH_2$ endowed
with the operator norm. By $\calL_2(\calH_1,\calH_2)$ we denote the Hilbert-Schmidt
operators from $\calH_1$ to $\calH_2$. For brevity we set $\calL(\calH):=\calL(\calH,\calH)$ and $\calL_2(\calH):=\calL_2(\calH,\calH)$.\par
We write $A \lesssim B$ to express that there exists a constant $C>0$ such that $A\leq C B$, and we write $A\eqsim B$ if
$A\lesssim B$ and $B\lesssim A$. Finally, for $\calH_1$ and $\calH_2$ Hilbert spaces we write $\calH_1\simeq \calH_2$ if $\calH_1$ and
$\calH_2$ are isomorphic as Hilbert spaces.
\end{notation}
\section{Preliminaries}\label{s:prelim}

In this section we recall shortly the definitions which are
necessary to formulate our main result. Throughout this section $H$ and $\calH$ will denote Hilbert spaces.
\subsection{Analytic semigroups}
A $C_0$-semigroup $(S(t))_{t\geq 0}$ on $\calH$ is a family of operators in $\calL(\calH)$ satisfying $S(0)=I$,
$S(t+s)=S(t)S(s)$ for all $s,t\geq 0$, and $t\mapsto S(t)x$ is continuous as $\calH$-valued function for all $x\in \calH$. The
generator $A$ of a semigroup $S$ is defined by
\begin{align*}
D(A)&:=\{x\in \calH\, : \, \lim_{t\downarrow 0} \tfrac{S(t)x-x}{t} \textrm{ exists in }\calH \};\\
Ax&:=\lim_{t\downarrow 0} \tfrac{S(t)x-x}{t}, \quad x\in D(A).
\end{align*}
Note that generally $D(A)\neq \calH$, although one necessarily has that $D(A)$ is dense in $\calH$. Conversely, given a densely
defined
closed operator $A$ on $\calH$, if the corresponding abstract Cauchy problem
\begin{align*}
\lk\{
\begin{array}{rll} \frac{d}{dt} u_x (t) &= Au_x(t), & t>0,
\\ u_x(0) &= x\in \calH,&
\end{array} \rk.
\end{align*}
has a unique (weak) solution for all $x\in \calH$ then $A$ is the generator of the semigroup defined by
$S(t)x:=u_x(t)$.\par
For $\theta \in [0,\pi]$ we define $\Sigma_{\theta}:= \{ z\in
\C\setminus\{0\} \,:\, |\arg(z)| < \theta\}$. We recall the
definition of an analytic $C_0$-semigroup, see also \cite[Chapter
2.5]{Pazy:83}.
\begin{definition}
Let $\theta\in (0,\frac{\pi}{2})$. A $C_0$-semigroup $(S(t))_{t\geq 0}$ on $\calH$ is called \emph{analytic
in} $\Sigma_{\theta}$ if
\begin{enumerate}
\item $S$ extends to an analytic function $S:\Sigma_{\theta}\rightarrow \calL(\calH)$;
\item $S(z_1+z_2)=S(z_1)S(z_2)$ for all $z_1,z_2\in \Sigma_{\theta}$;
\item $\lim_{z\rightarrow 0; z\in \Sigma_\theta}S(z)x=x$ for all $x\in \calH$.
\end{enumerate}
\end{definition}
If there exist constants $\theta\in (0,\frac{\pi}{2})$, $\omega \in \R$ and $K>0$ such that
$\omega+\Sigma_{\frac{\pi}{2}+\theta}\subseteq \rho(A)$ and for every $\lambda\in
\omega+\Sigma_{\frac{\pi}{2}+\theta}$ one has
$$\lk |\lambda-\omega|\|R(\lambda :A)\rk\|_{\CL
(\calH)}\le K,
$$
then the semigroup generated by $A$ is analytic on $\Sigma_{\theta'}$ for all $0<\theta'<\theta$ (see Pazy
\cite[Theorem 2.5.2]{Pazy:83}, also for a reverse statement).
This justifies the following definition:
\begin{definition}\label{d:uniftype}
Let $A$ be the generator of an
analytic $C_0$-semigroup on $\calH$. We say that $A$ is \emph{of type} $(\omega,\theta,K)$, where $\omega\in\R$, $\theta\in
(0,\frac{\pi}{2})$ and $K>0$, if
$\omega+\Sigma_{\frac{\pi}{2}+\theta} \subseteq \rho(A)$ and
$$|\lambda-\omega|\n R(\lambda:A) \n_{\CL (\calH)} \leq K \quad
\textrm{for all } \lambda \in
\omega+\Sigma_{\frac{\pi}{2}+\theta}.$$ Moreover, we say that a semigroup $S$ is \emph{of type} $(\omega,
\theta,K)$ if its generator is of type $(\omega,\theta,K)$.
\end{definition}
\begin{example}\label{ex:saOper}
Let $-A:D(A)\rightarrow \calH$ be $m$-$\theta$-accretive for some $\theta\in [0,\frac{\pi}{2})$; i.e.\ $1 \in
\rho(A)$ and for all $x\in D(A)$ with $\n x\n=1$ we have
$$ \langle -A x ,x \rangle_{\calH} \in \overline{\Sigma_{\theta}}.$$
Suppose moreover that $A$ is injective. Then for any $\omega \geq 0$ the operator $A+\omega$ is analytic of type
$$(\omega(1+2(\cos\theta')^{-1}),\theta',(4+\tfrac{4}{\sqrt{3}})(1-\sin(\theta+\theta'))^{-1})$$ for all $\theta'\in
(0,\frac{\pi}{2}-\theta)$.
More specifically, if $A$ is self-adjoint and $A \leq 0$ (i.e.\ $\langle A
x,x\rangle_{\calH} \leq 0$ for all $x\in \calH$), then $A+\omega$ is analytic of type
$(\omega(1+2(\cos\theta')^{-1}),\theta',2(1-\sin\theta')^{-1})$ for all $\theta'\in (0,\frac{\pi}{2})$.
\end{example}
\begin{proof}
We first prove this example for the case that $\omega=0$. Suppose $-A$ is $m$-$\theta$-accretive and injective.
By \cite[Proposition 7.1.1]{haase} the operator $A$ generates an analytic $C_0$-semigroup on $\Sigma_{\theta'}$ for all
$\theta'\in (0,\frac{\pi}{2}-\theta)$. Moreover, by \cite[Corollary 2.1.17]{haase} we have that for $\lambda \in
\Sigma_{\pi-\theta}$ the following estimate holds:
\begin{align*}
\n \lambda R(\lambda:A) \n & \leq (2+\tfrac{2}{\sqrt{3}})\sup_{z\in -\overline{\Sigma_{\theta}}}\Big|
\frac{z}{\lambda-z}\Big|\\
& \leq (2+\tfrac{2}{\sqrt{3}})\cdot \left\{ \begin{array}{ll}
1;& |\arg(\lambda)|\leq \frac{\pi}{2}-\theta;\\
\big(1+\cos(\theta-|\arg(\lambda)|)\big)^{-1}; & |\arg(\lambda)|>\frac{\pi}{2}-\theta.
\end{array}
\right.
\end{align*}
Thus for $\theta'\in (0,\frac{\pi}{2}-\theta)$ fixed we find that for all $\lambda\in \Sigma_{\frac{\pi}{2}+\theta'}$ we
have $$\n \lambda R(\lambda:A)\n \leq (2+\tfrac{2}{\sqrt{3}})(1-\sin(\theta+\theta'))^{-1}.$$\par
If $A$ is self-adjoint and $A\leq 0$ then by \cite[Corollary 7.1.6]{haase} we have that $A$ generates an analytic
semigroup and
\begin{align*}
\n \lambda R(\lambda:A)\n & \leq \sup_{t\in (-\infty,0)} \big| \tfrac{\lambda}{\lambda-t}\big|\\
& \leq \left\{\begin{array}{ll}
1;& |\arg(\lambda)|\leq \frac{\pi}{2};\\
\big(1+\cos \arg(\lambda)\big)^{-1}; & |\arg(\lambda)|>\frac{\pi}{2}.
\end{array}
\right.
\end{align*}\par
Now assume $\omega>0$. As $A$ generates an analytic semigroup $S$, it follows that $A+\omega$ generates the analytic
semigroup $(e^{\omega t}S(t))_{t\geq 0}$ with the same angle of analyticity, and because $ \omega + \Sigma_{\frac{\pi}{2}+\theta'} \subset \rho(A+\omega)$ clearly also
$ \omega(1+2(\cos\theta')^{-1})+\Sigma_{\frac{\pi}{2}+\theta'} \subset \rho(A+\omega)$ for any $\theta'\in
(0,\frac{\pi}{2}-\theta)$. As for the estimate on the resolvent; fix $\theta'\in (0,\frac{\pi}{2}-\theta)$. One may
check that for $\lambda \in \omega(1+2(\cos\theta')^{-1})+\Sigma_{\frac{\pi}{2}+\theta'}$ we have $|\lambda|\geq 2\omega$, whence
$\frac{|\lambda|}{|\lambda-\omega|}\leq 2$. Let $\lambda \in \omega(1+2(\cos\theta')^{-1})+\Sigma_{\frac{\pi}{2}+\theta'}$,
then
\begin{align*}
\n \lambda R(\lambda:A+\omega) \n & = \tfrac{|\lambda|}{|\lambda-\omega|}\n (\lambda-\omega) R(\lambda-\omega:A) \n\\
& \leq (4+\tfrac{4}{\sqrt{3}})(1-\sin(\theta+\theta'))^{-1},
\end{align*}
or, for the self-adjoint case,
\begin{align*}
\n \lambda R(\lambda:A+\omega) \n & = \tfrac{|\lambda|}{|\lambda-\omega|}\n (\lambda-\omega) R(\lambda-\omega:A) \n\\
& \leq 2(1-\sin(\theta+\theta'))^{-1}.
\end{align*}
\end{proof}\par
Our perturbation result involves estimates in certain abstract interpolation
and extrapolation spaces of $D(A)$. In applications, these spaces often correspond to the Sobolev spaces, as we will
explain below. Interpolation and extrapolation spaces of $D(A)$ may be defined for any operator $A$ that generates an
analytic semigroup of type
$(\omega,\theta,K)$ on $\calH$ for some $\omega\in \R$, $\theta\in (0,\frac{\pi}{2})$ and $K>0$. We define
the
extrapolation spaces of $A$ conform to \cite[Section 2.6]{Pazy:83};
i.e.\ for $\delta>0$ and $\lambda \in \C$ such that
$\mathscr{R}e(\lambda)>\omega$ we define $\calH^{A}_{-\delta}$ to be
the closure of $\calH$ under the norm $\n x\n_{\calH^{A}_{-\delta}} := \n
(\lambda I -A)^{-\delta} x \n_\calH$. We also define the fractional
domain spaces of $A$, i.e.\ for $\delta>0$ we define
$\calH^{A}_{\delta}=D((\lambda I -A)^{\delta})$ and $\n
x\n_{\calH^{A}_{\delta}} := \n (\lambda I -A)^{\delta} x \n_\calH$. For the definition of the fractional powers of $\lambda
I - A$ we refer to \cite[Chapter 2]{Pazy:83}. One
may check that regardless of the choice of $\lambda$ the
extrapolation spaces and the fractional domain spaces are uniquely
determined up to isomorphisms: for $\delta>0$ one has $(\lambda I
-A)^{\delta}(\mu I -A)^{-\delta}\in \CL (\calH)$ and
$$\n (\lambda I -A)^{\delta}(\mu I -A)^{-\delta}\n_{\CL (\calH)}\leq C(\omega,\theta,K,\lambda,\mu),$$
where $C(\omega,\theta,K,\lambda,\mu)$ denotes a constant
depending only on $\omega,\theta,K,\lambda,$ and $\mu$. Moreover,
for $\delta,\, \beta\in \R$ one has $(\lambda
I-A)^{\delta}(\lambda I-A)^{\beta}=(\lambda I-A)^{\delta+\beta}$
on $\calH^A_{\gamma}$, where $\gamma= \max\{\beta,\delta+\beta\}$ (see
\cite[Theorem 2.6.8]{Pazy:83}).\par
In the case that $A$ is our differential operator from Example
\ref{example}, the scale of interpolation and extrapolation
spaces coincides with the scale of Sobolev spaces. More precisely,
if $\calH=L^2(0,1)$, then for $\delta\in (\frac{3}{4},1)$ we have $\calH_\delta^A=W^{2\delta,2}_0(0,1)$ and for $\delta\in
(0,\frac{3}{4})$ we have $\calH_{\delta}^{A}=W^{2\delta,2}(0,1)$ (see also Section \ref{s:finEl}).\par
\subsection{Parabolic SPDEs}\label{ss:setting}
Let $H$ be a Hilbert space and let $W_H$ be an $H$-cylindrical Brownian motion on $(\Omega,(\mathcal{F}_t)_{t\geq 0},\P)$. Recall from the introduction that we wish to approximate the solution to the following type of stochastic differential equation set in a Hilbert space $\calH$:
\begin{equation}\label{SDE} \tag{SDE}
\left\{ \begin{aligned} dU(t) & = AU(t)\,dt + F(t,U(t))\,dt + G(t,U(t))\,dW_H(t);\quad t\in [0,T],\\
U(0)&=x_0. \end{aligned}\right.
\end{equation}
Here $A$, $F$ and $G$ are assumed to satisfy:
\begin{itemize}
\item[\MA{}]
The operator $A$ is the generator of an analytic $C_0$-semigroup $S$ on $\calH$ of type $(\omega,\theta,K)$ for
some $\omega\in \R$, $\theta \in (0,\frac{\pi}{2})$, $K>0$.
\item[\MF{}] For some $\theta_F>  -1$ the function
$F:[0,T]\times \calH\rightarrow \calH_{\theta_F}^{A}$ is uniformly Lipschitz continuous and uniformly of linear growth on
$\calH$.
That is to say, there exist constants $C_0$ and
$C_1$ such that for all $t\in [0,T]$ and all $x,y\in \calH$ one has
\begin{align*}
\n F(t,x) - F(t,y) \n_{\calH_{\theta_F}^{A}} & \leq C_0 \n x-y\n_{\calH}; &\textrm{and}&& \n F(t,x)\n_{\calH_{\theta_F}^{A}} &\leq
C_1(1+\n x\n_{\calH}).
\end{align*}
The least constant $C_0$ such that the above holds is denoted by Lip$(F)$, and the least constant $C_1$ such that the
above holds is denoted by $M(F)$.\par
Moreover, for all $x\in \calH$ we have that $F(\cdot,x):[0,T]\rightarrow \calH^{A}_{\theta_F}$ is measurable.
\item[\MG{}] For some  $\theta_G>-\inv{2}$ the function $G : [0,T]\times \calH\rightarrow \calL_2(H,\calH_{\theta_G}^{A})$
is uniformly Lipschitz continuous and uniformly of linear growth on
$\calH$. We denote the corresponding constants by Lip$(G)$ and $M(G)$.\par
Moreover, we have that the mapping $G(\cdot,x)h:[0,T]\rightarrow \calH_{\theta_G}^A$ is measurable for all $x\in \calH$ and all $h\in H$.
\end{itemize}
Under these conditions it is well-known (see \cite[Theorem 7.4]{DaPZab:92} for the case $\theta_F=\theta_G=0$,
\cite{Brze:97} for the general case) that provided $x\in
L^p(\Omega,\mathcal{F}_0,\calH)$ for some $p >  2$ such that $\inv{p}<\inv{2}+\theta_G$ there exists a unique adapted process $U\in L^p(\Omega,C([0,T],\calH))$
such that for all $t\in [0,T]$ it holds that
\begin{align}\label{varcons}
U(t)&=S(t)x_0+\int_{0}^{t}S(t-s)F(s,U(s))\,ds+\int_{0}^{t}S(t-s)G(s,U(s))\,dW_H(s), \quad \textrm{a.s.}
\end{align}
This process is referred to as the \emph{mild solution} to \eqref{SDE}.
\par
Returning to Example \ref{example} we find that conditions \MA{}, \MF{} and \MG{} are
satisfied provided $f$ and $g$ are measurable and there exists a $C$ such that for all $t\in [0,T]$ and $x,y\in \R$ we have
\begin{align*}
| f(t,x) - f(t,y)| & \leq C|x-y| &\textrm{and} && | g(t,x) - g(t,y)| & \leq C|x-y|;\\
| f(t,0) | & \leq C & \textrm{and} && | g(t,0) | & \leq C.
\end{align*}
We define the operator $F:[0,T]\times L^2(0,1)\to  L^2(0,1)$ to be the associated Nemytski operator given for $t\in
[0,T]$ and $u\in
L^2(0,1)$ by $$ F(t,u)(\xi) = f(t,u(t,\xi)),\quad \xi \in [0,1],$$ and $G:[0,T]\times L^2(0,1)\to
\calL_2(L^2(0,1),W^{-\inv{4}-\eps,2}(0,1))$ by $$[G(t,u)h](\xi) = g(t,u(t,\xi))h(\xi), \quad \xi\in [0,1],$$ for
$u\in L^2(0,1)$, $t\in [0,T]$ and $h\in
L^2(0,1)$. As
mentioned in the introduction, one may show that $F$ satisfies \MF{} with $\theta_F= 0$ and $G$ satisfies \MG{} with
$\theta_G=-\inv{4}-\eps$.\par

\subsubsection{A perturbation result}

Let $\calH_0$ be a subspace of $\calH$ which
may be finite dimensional. Let $P_0\in \calL(\calH,\calH_0)$ be a bounded projection of $\calH$ onto $\calH_0$. Let $i_{\calH_0}$
represent the canonical embedding of $\calH_0$ into $\calH$ (note however
that we shall omit $i_{\calH_0}$ when its usage is clear from the
context).\par

Let $A_0$ be the generator of an analytic
$C_0$-semigroup $S_0$ on $\calH_0$. Consider the equation \eqref{SDE} with $A$, $F$ and $G$ satisfying \MA{}, \MF{}, and \MG{}. Let $x_0\in L^p(\Omega,\mathcal{F}_0,\calH)$
for some $p\geq 2$ and let $U$ be the solution to \eqref{SDE}. In \cite{coxHau:11} we have
shown the following abstract result:
\begin{theorem}\label{t:app}
Let $\omega\geq 0$, $\theta\in(0,\frac{\pi}{2})$ and $K>0$ be such that $A$ and $A_0$ are both of type
$(\omega,\theta,K)$.  Suppose there exists
a constant $p\in (2,\infty)$ and a constant $\delta\in [0,1]$ such that
$$\delta < \min\{1+\theta_F, \tinv{2}+\theta_G-\tinv{p}\}$$
and such that for some  $\lambda_0\in \C$, $\mathscr{R}e(\lambda_0)> \omega$ we have
\begin{align}\label{app:ass}
D_{\delta}(A,A_0):=  \n R(\lambda_0:A)-i_{\calH_0} R(\lambda_0:A_0)P_0 \n_{\calL(\calH_{\delta-1}^{A},\calH)} < \infty.
\end{align}
Suppose $x_0\in L^p(\Omega,\calF_0,\calH_{\delta}^{A})$. Then there exists a unique adapted process $U^{(0)}\in
L^p(\Omega,C([0,T],\calH_0))$ satisfying, for all $t\in [0,T]$,
\begin{equation}
\begin{aligned}\label{pert.sol}
U^{(0)}(t) = S_0(t) x_0 &+ \int_{0}^{t} S_0(t-s)P_0F(s,U^{(0)}(s))\,ds\\
&  + \int_{0}^{t}
S_0(t-s)P_0 G(s,U^{(0)}(s))\,dW_H(s), \quad \textrm{a.s.}
\end{aligned}
\end{equation}
Moreover, there exists constant
\begin{align}\label{eq:depC}
C = C\big(\omega,\theta,K,\textrm{Lip}(F), \textrm{Lip}(G),M(F),M(G),\n P_0\n_{\calL(\calH,\calH_0)},1+D_{\delta}(A,A_0)\big)
\end{align}
such that
\begin{equation}\label{pert_est}
\begin{aligned}
\n U - i_{\calH_0}U^{(0)} \n_{L^p(\Omega,C([0,T],\calH))} & \leq C  D_{\delta}(A,A_0)(1+\n x_0
\n_{L^p(\Omega,\calH_{\delta}^{A})}).
\end{aligned}
\end{equation}
Here $C$ may be chosen such that it depends continuously on all the parameters on the right-hand side of \eqref{eq:depC} and does not depend on any other parameters.
\end{theorem}
\begin{remark}
In \cite{coxHau:11} this result is in fact proven in the more general setting of \textsc{umd} Banach spaces. The class of \textsc{umd} Banach spaces has been introduced by Burkholder (see \cite{Bu3} for an overview), and includes all Hilbert spaces and most reflexive `classical' function spaces such as the $L^p$-spaces, $p\in (1,\infty)$. In \cite{coxHau:11} the perturbation theorem is used to obtain convergence rates for the Yosida approximation in the \textsc{umd} setting.
\end{remark}

\begin{remark}\label{r:unif_analytic}
An important issue to keep in mind when using this theorem to prove convergence of approximations, is that the constant $C$
in \eqref{eq:depC} and \eqref{pert_est} depends on $\omega,\theta,$ and $K$. Thus given a sequence of approximating operators $(A_n)_{n\in\N}$
for $A$ it is necessary that they are all of type $(\omega,\theta,K)$ for some fixed $(\omega,\theta,K)$. We call this
property \emph{uniform analyticity of }$(A_n)_{n\in\N}$.\par
By Example \ref{ex:saOper}, if $(A_n)_{n\in\N}$ are all self-adjoint then they are uniformly analytic. More generally, the operators $(A_n)_{n\in\N}$
are uniformly analytic if they are all injective and $m$-$\theta$-accretive for some $\theta\in (0,\frac{\pi}{2})$.\par
Finally, note that the implied constant in \eqref{pert_est} depends on $1+D_{\delta}(A,A_0)$ and on $\n P_0\n_{\calL(\calH,\calH_0)}$. In general this does not cause any difficulties: given a sequence of approximating operators $(A_n)_{n\in\N}$ defined on subspaces $(\calH_n)_{n\in\N}$ of $\calH$, one usually takes $P_0$ to be the orthogonal projection of $\calH$ onto $\calH_n$ whence $\n P_n\n_{\calL(\calH,\calH_0)}=1$ for all $n\in \N$. Moreover, it is necessary that $D_{\delta}(A,A_n)\downarrow 0$ as $n\rightarrow \infty$ in order to obtain convergence, whence in particular $1+D_{\delta}(A,A_n)$ is uniformly bounded in $n$.
\end{remark}
\begin{remark}
In the next two sections we will demonstrate that estimates of the type \eqref{app:ass} are quite natural for various
types of `approximating' operators $A_0$. Note that if it is possible to find estimates for $D_{\delta_1}(A,A_0)$ and $D_{\delta_2}(A,A_0)$, $0\leq \delta_1 \leq \delta_2 \leq 1$, then 
estimates for the intermediate values $D_{\delta}(A,A_0)$, $\delta\in (\delta_1,\delta_2),$ may generally be obtained by interpolation.
This is demonstrated in Section \ref{s:finEl}.
\end{remark}

\section{Spectral Galerkin method}\label{s:galerkin}
The relevance of Theorem \ref{t:app} for proving convergence of approximation schemes is neatly demonstrated when considering spectral Galerkin methods\par
Consider the equation \eqref{SDE} under the assumptions \MA{}, \MF{} and \MG{} with the additional assumption that $A$ is a
self-adjoint operator
generating an eventually compact semigroup on a Hilbert space $\calH$ (see \cite[Definition II.4.23]{EngNa:00}). By \cite[Corollary V.3.2 and Section IV.1]{EngNa:00} it follows that the
spectrum of $A$ consists
only of eigenvalues, and these eigenvalues lie in $(-\infty,\omega]$ for some $\omega\in \R$. We denote the eigenvalues (which are listed in algebraic multiplicity)
by
$(\lambda_n)_{n\in\N}$, and assume $(\lambda_n)_{n\in\N}$ is ordered such that $\lambda_{n+1}\leq
\lambda_{n}$ for all $n\in\N$. Let $(\phi_n)_{n\in\N}$ be the
eigenfunctions corresponding to $(\lambda_n)_{n\in\N}$ (picked such that they are orthogonal), and define
$\calH_n= \textrm{span}\{\phi_1,\ldots,\phi_n\}$. Let $P_n\in
\calL(\calH,\calH_n)$ be given by $P_n x = \sum_{k=1}^n \langle x,\phi_k \rangle_{\calH} \phi_k$ for $x\in \calH$ (thus $P_n$ is the
orthogonal
projection of $\calH$ onto $\calH_n$).\par
Let $U$ be the solution to
\eqref{SDE} with $A$ as described above and initial data $x_0\in
L^p(\Omega,\calF_0,\calH)$. Let $U^{(n)}$ be the $n^{\textrm{th}}$
Galerkin approximation; i.e.\ $U^{(n)}$ is the solution to the finite-dimensional problem in $\calH_n$:
\begin{equation}\label{galerkin_proj}
\begin{aligned}
 U^{(n)}(t) & = P_n x_0 + \int_{0}^{t} P_n A
U^{(n)}(s)\,ds + \int_{0}^{t} P_n F(s, U^{(n)}(s))\,ds \\[\medskipamount]
& \quad + \int_{0}^{t} P_n
G(s,U^{(n)}(s)) \,dW_H(s),\quad t\in [0,T].
\end{aligned}
\end{equation}
Note that by setting $U_n(t)=\sum_{j=1}^{n}u_j(t)\phi_j$, with $u_j(t)\in L^p(\Omega,C([0,T]))$, and testing against $\phi_k$, $k=1,\ldots,n$, this reduces to equation \eqref{galerkin_n} in the introduction.
Theorem \ref{t:app} leads to the following convergence result.
\begin{proposition}\label{prop:galerkin}
For any $\eta\in [0,1]$ and $p\in (2,\infty)$ such that
$$\eta<\min\{1+\theta_F,\tfrac{1}{2}+\theta_G-\tinv{p} \}$$ we have,
assuming $x_0\in L^p(\Omega,\mathcal{F}_0,\calH_{\eta}^A)$, that a solution to \eqref{galerkin_proj} exists and
\begin{align*}
\n U - U^{(n)}\n_{L^p(\Omega,C([0,T],\calH))}
&\lesssim  |\lambda_{n+1}|^{-\eta}(1+\n x_0
\n_{L^p(\Omega,\calH_{\eta}^{A})}),
\end{align*}
with implied constants independent of $n$, $(\lambda_n)_{n\in\N}$ and $x_0$.
\end{proposition}
\begin{proof}[Proof of Theorem \ref{c:galerkin_pathwise}.]
This Theorem is a direct consequence of the Borel-Cantelli Lemma and Proposition \ref{prop:galerkin} above (see \cite[Lemma 2.1]{Kloeneuen} for the precise argument).
\end{proof}
\begin{proof}[Proof of Proposition \ref{prop:galerkin}.]
Let $\omega\geq 0$ be such that $\lambda_1<\omega$. Then $A-\omega$ is self-adjoint and $A-\omega\leq
0$. Thus by Example \ref{ex:saOper} $A$ is analytic of type $(\omega(1+2(\cos\theta')^{-1}),\theta',2(1-\sin\theta')^{-1})$
for all $\theta'\in (0,\frac{\pi}{2})$. Define $A_n:\calH_n\rightarrow
\calH_n$ by
\begin{align*}
A_n & = \sum_{k=1}^{n} \lambda_k \langle \cdot, \phi_k \rangle_{\calH}
\phi_k,
\end{align*}
i.e.\ $A_n= P_n A i_{\calH_n}$, where $i_{\calH_n}$ is the canonical
embedding of $\calH_n$ into $\calH$. Clearly $A_n-\omega$ is again self-adjoint and $A_n-\omega\leq
0$. Thus by Example \ref{ex:saOper} $A_n$ is analytic of type
$(\omega(1+2(\cos\theta')^{-1}),\theta',2(1-\sin\theta')^{-1})$
for all $\theta'\in (0,\frac{\pi}{2})$. It follows that $(A_n)_{n\in\N}$ is uniformly analytic.\par
Note that by the equivalence of strong and mild solutions in the finite-dimensional case the process $U^{(n)}$ satisfying \eqref{galerkin_proj} is precisely the solution to \eqref{pert.sol} in
Theorem \ref{t:app} if we take $\calH_0=\calH_n$, $P_0=P_n$ and $A_0=A_n$.\par
In order to apply Theorem \ref{t:app}, we must prove that condition \eqref{app:ass} holds for some
appropriate $\delta$. Fix $\lambda\in \rho(A)$ such that $\Re
e\lambda>\omega$. We have
\begin{align*}
R(\lambda:A)-i_{\calH_n} R(\lambda:A_n) P_n & =
\sum_{k=n+1}^{\infty} \frac{\langle
\cdot,\phi_k\rangle_{\calH}}{\lambda-\lambda_k}\phi_k
\end{align*}
and for $\delta\geq 0$ and $x\in \calH_{\delta}^A$ we have
\begin{align*}
(\lambda I -A)^{\delta}x & = \sum_{k=1}^{\infty}
(\lambda-\lambda_k)^{\delta}\langle x,\phi_k\rangle_{\calH}.
\end{align*}
As $|\lambda-\lambda_{i+1}| \geq |\lambda-\lambda_{i}|$ for all
$i\in \N$ we have, for all $\delta \in [0,1)$,
\begin{align*}
\n R(\lambda:A) - i_{\calH_n} R(\lambda:A_n) P_n \n_{\CL
(\calH_{\delta-1}^{A},\calH)} &\eqsim \Big\n \sum_{i=n+1}^{\infty}
(\lambda-\lambda_i)^{-\delta} \langle \,\cdot\, , \phi_i\rangle_{\calH} \phi_i \Big\n_{\calL(\calH)} \\
& \leq |\lambda-\lambda_{n+1}|^{-\delta}.
\end{align*}
with implied constants depending on $A$ only in terms of
$\omega,\theta,$ and $K$.\par 
Let $\eta\in [0,1]$ satisfying
$\eta<\min\{1+\theta_F,\frac{1}{2}+\theta_G-\inv{p}\}$ be given and fix
$T>0$. The desired result now follows by applying Theorem
\ref{t:app} with $\delta=\eta$, $\calH_0=\calH_n$, $P_0=P_n$, $A_0=A_n$, and $U^{(0)}=U^{(n)}$.
\end{proof}
\section{Finite elements}\label{s:finEl}
It is not our intention to go into great detail concerning the question how a finite element method
is constructed, nor to state convergence results for finite element methods in general. Instead, we wish to demonstrate
by means of an example that the estimate necessary for the application of
Theorem \ref{t:app}, namely estimate \eqref{app:ass} for $\delta=1$, is precisely the type of estimate sought after when trying to
prove convergence of finite element methods for time-independent problems. The example we consider is the case that $A$
is a second order differential operator\sgc{} with homogeneous Dirichlet boundary conditions\cgs{}. We follow the approach of \sgc{}\cite[Sections 3.1 and 3.2]{ErnGuermond:2004}. More specifically,
we consider the setting of~\cite[Theorem 3.18]{ErnGuermond:2004}, which we will now provide.\par
Let $d\in \N$ and let $D\subset \R^d$ be a convex polyhedron (see~\cite[Definition 1.47]{ErnGuermond:2004}).
Let the spaces $W^{2s,2}_{B}(D)$, \sgc{}$s\in [0,\infty)\setminus\{ \frac{1}{4} \}$ be given by
\begin{align*}
W^{2s,2}_{B}(D) & :=\left\{ \begin{array}{ll} W^{2s,2}(D), & 0\leq s < \sgc{}\frac{1}{4}\cgs{};\\
\sgc{}\{ u\in W^{2s,2}(D)\,:\, u|_{\partial D} = 0 \}, & s>\sgc{}\frac{1}{4}\cgs{}.
\end{array} \right.
\end{align*}
Consider\cgs{} the equation \ref{SDE} for the case that $\calH=L^2(D)$ and $A:\sgc{}W^{2,2}_B(D)\cgs{}\rightarrow L^2(D)$
is a second-order elliptic operator defined by
\begin{align}\label{defA}
Au:= \sum_{i,j=1}^{n} \frac{\partial}{\partial x_j} \Big( a_{ij} \frac{\partial u}{\partial x_i}\Big),\quad u\in
\sgc{}W^{2,2}_B(D)\cgs{},
\end{align}
where the functions \sgc{}$a_{ij}\in C^1(\overline{D})$\cgs{} satisfy \sgc{}$a_{i,j} = a_{j,i}$ and\cgs{} the ellipticity condition, i.e.\ there exists
an $\alpha>0$ such that for all $x\in D$ and all $\xi\in\R^n$ we have
\begin{align*}
\sum_{i,j=1}^{n} a_{ij}(x)\xi_i\xi_j & \geq \alpha \sum_{i=1}^{n} |\xi_i|^2.
\end{align*}
\sgc{}Our choice of the domain of $A$ implies that we assume homogeneous Dirichlet boundary conditions. \cgs{}
% 
% Moreover, we assume boundary conditions as posed in \cite[Section 5.6]{brenner}; i.e.\ we assume $Bu\equiv 0$ on
% $\partial D$, where we define $B$ by
% \begin{align}\label{eq:defBound}
% Bu(x) := \sum_{i,j=1}^{n} a_{ij}(x)\frac{du}{dx_i}\nu_j(x),
% \end{align}
% where $x\in \partial D$ and $\nu(x)$ is the normal of $\partial D$
% in $x$. This is well-defined provided $u\in W^{2s,2}(D)$ for $s>\frac{3}{4}$.\par
% To incorporate these boundary conditions in the domain of $A$ we define the space
% $$ W^{2,2}_{B}(D):= \{ u\in W^{2,2}(D)\,:\, Bu=0 \}. $$
% 
Note that $A:\sgc{}W^{2,2}_B(D)\cgs{}\rightarrow L^2(D)$ is self-adjoint and $A\leq 0$,
whence by Example \ref{ex:saOper} $A$ is analytic of type $(0,\theta,2(1-\sin\theta)^{-1})$ for all $\theta\in
(0,\frac{\pi}{2})$, \sgc{}moreover, it follows from~\cite[Corollary 4.3.6]{Lunardi:09} and~\cite[Section 10.2]{Calderon:1964}
that 
\begin{align}\label{fracPowEllip}
 \calH^A_{(1-s)\alpha + s\beta }\simeq [\calH^{A}_{\alpha},\calH^{A}_{\beta}]_{s}.
\end{align}
for all $s\in [0,1]$, $\alpha,\beta\in \R$ such that $\alpha \beta\geq 0$
($[\calH_1,\calH_2]_{s}$ denotes the complex interpolation 
space of the Hilbert spaces $\calH_1$ and $\calH_2$ with parameter
$s\in (0,1)$). In addition, by\cgs{} \cite[Section
8]{Gris:67} one has, for $s\in [0,1] \setminus \frac{1}{4}$,
\begin{align}\label{sobolevinterpol}
 [L^2(D),\sgc{}W^{2,2}_B(D)\cgs{}]_{s} \simeq W^{2s,2}_{B}(D).
\end{align}
(The interested reader is referred to \cite{Gris:67} for the special case that $s=\frac{1}{4}$.) 
\par
\sgc{}Let $(\hat{K},\hat{P},\hat{\Sigma})$ be a reference Lagrange finite element in $\R^d$ such that $\hat{K}$ is a polyhederon and $\hat{P}$ is the set of 
all polynomials of degree 1 on $\hat{K}$ (see~\cite[Definitions 1.23 and  1.27]{ErnGuermond:2004}). Furthermore, let $I\subseteq (0,\infty)$ and let $(\mathcal{T}_h)_{h\in I}$ be a shape-regular family of $(\hat{K},\hat{P},\hat{\Sigma})$-geometrically conformal
meshes of $D$ (see~\cite[Definition 1.55 and 1.107]{ErnGuermond:2004}, \emph{shape-regular} means, roughly speaking, that the 
angles between sides of the mesh elements cannot get arbitrarily small) satisfying $\max_{K\in \mathcal{T}_h} \operatorname{diam}(K) \leq h$ for all $h\in I$. Let $(V_h)_{h\in I}$ be the corresponding finite-element space of continuous piecewise linear functions $v_h$ satisfying $v_h|_{\partial D} =0$. \cgs{} Set $X_h:=(\{f\in V_h\},\n \cdot \n_{L^2(\sgc{}D\cgs{})})$. (Note that in the introduction the notation $V_n$ was used for clarity where, strictly speaking, one should have written $V_{\inv{n}}$.)\par
Let $a:W^{1,2}_{B}(D)\times W^{1,2}_{B}(D) \rightarrow \R$ be the form
associated with $A$ (i.e.\ $a(u,v)=\langle Au,v\rangle_{L^2(D)}$ for all $u\in D(A)$ and $v\in W^{1,2}_{B}(D)$). For $h\in I$
fixed the finite-element approximation $U^{(h)}$ of $U$, the solution to \eqref{SDE} with initial condition
$u_0$, is the element of $L^p(\Omega,C([0,T],X_h))$ satisfying, for all $t\in[0,T]$,
\begin{equation}\label{finelsol}
\begin{aligned}
\langle U^{(h)}(t) , v_h \rangle_{L^2(D)} & = \int_{0}^{t} a(U^{(h)}(s),v_h)\,ds + \int_{0}^{t} \langle
F(s,U^{(h)}(s)), v_h\rangle_{L^2(D)} \,ds\\
&\quad + \int_{0}^{t} G^*(s,U^{(h)}(s)) v_h \,dW_H(s), \qquad \textrm{a.s.\ for all } v_h\in V_h,\\
\langle U^{(h)}(0) , v_h \rangle_{L^2(D)} & = \langle u_0, v_h \rangle_{L^2(D)}, \qquad \textrm{a.s.\ for all } v_h\in V_h.
\end{aligned}
\end{equation}
Here $G^*(s,U^{(h)}(s))\in \calL(L^2(D),H)$ is the adjoint of $G(s,U^{(h)}(s)) \in \calL(H,L^2(D))$. Of course it suffices to check the above for $(v_h^{(k)})_{k=1}^{N}$ a basis of $V_h$. Also note that $G(s,U^{(h)}(s))\in \calL_2(H,\calH)$, whence its dual operator $G^*(s,U^{(h)}(s))\in \calL_2(\calH,H)$. I.e., $s\mapsto G^*(s,U^{(h)}(s)) v_h$ can be interpreted as an $H'$-valued process, whence the stochastic integral is real-valued.\par
\begin{proposition}\label{prop:finEl}
Consider the equation \eqref{SDE} in the Hilbert space $L^2(D)$ with $A$ as defined in \eqref{defA} and $F$, $G$ satisfying \MF{} and
\MG{} with $\calH=L^2(D)$. 
Let $p\in (2,\infty)$ and $\eta\in [\inv{2},1]$ satisfy $$\tinv{2} \leq \eta <
\min\{1+\theta_F,\tinv{2}+\theta_G-\tinv{p}\}.$$ Let $u_0 \in W^{2\eta,2}_{B}(D)$. Then there exists a unique $U^{(h)}\in
L^p(\Omega,C([0,T];\calH))$ satisfying \eqref{finelsol}. Moreover, for all
$h\in I$ we have
\begin{align*}
\n U - U^{h} \n_{L^p(\Omega,C([0,T];L^2(D)))} & \lesssim h^{2\eta} (1+ \n u_0\n_{W^{2\eta,2}_{B}(D)}),
\end{align*}
with implied constant independent of $h$ and $u_0$.
\end{proposition}
\begin{proof}[Proof of Theorem \ref{c:finEl_pathwise}.]
Note that in Theorem \ref{c:finEl_pathwise} the index set $I$ is given by $I=\{\inv{n}\,:\,n\in \N\}$ (and we write $U^{(n)}$ instead of $U^{\inv{n}}$ for notational clarity). \sgc{}Thus\cgs{} Theorem \ref{c:finEl_pathwise} follows directly from Proposition \ref{prop:finEl} and the Borel-Cantelli lemma.
\end{proof}
\begin{proof}[Proof of Proposition \ref{prop:finEl}.]
We need to rewrite the setting into the setting of Theorem \ref{t:app}. Fix $h\in I$. We define
$A_h:V_h\rightarrow V_h$ by $\langle A_h u, v\rangle_{L^2(D)} = a(u,v)$ for all $v\in V_h$. Note that $A_h$ is self-adjoint
because $A$ is self-adjoint, hence by Example \ref{ex:saOper} the operator $A_h$ is of type
$(0,\theta,2(1-\sin\theta)^{-1})$ for all $\theta\in (0,\frac{\pi}{2})$. In other words, the family of operators $(A_h)_{h\in I}$ is
uniformly analytic. We  define
$P_h:L^2(D)\rightarrow X_h$ to be the orthogonal projection of $L^2(D)$ onto $X_h$.\par
By equivalence of weak and mild solutions in finite dimensions, a process
$U^{(h)}\in L^p(\Omega,C([0,T],V_h))$ satisfies \eqref{finelsol} if and only if it satisfies, for all $t\in
[0,T]$,
\begin{align*}
U^{(h)}(t) = e^{t A_h}u_0 & + \int_{0}^{t} e^{(t-s)A_h}
P_h F(s,U^{(h)}(s))\,ds\\
& + \int_{0}^{t} e^{(t-s)A_h} P_h G(s,U^{(h)}(s))\,dW_H(s).
\end{align*}
In other words, $U^{(h)}=U^{(0)}$ in Theorem \ref{t:app} with $A_0=A_h$, $\calH_0=V_h$ and $P_0=P_h$.\par

It remains to prove an estimate of the type \eqref{app:ass}. \sgc{}Indeed, we 
are precisely in the setting of~\cite[Theorem 3.16, Remark 3.17 and Theorem 3.18]{ErnGuermond:2004}, which 
in combination with~\cite[Theorem 3.2.1.2]{Grisvard:2011} ensures
that there exists a constant $c$ such that for all $f\in L^2(D)$, $h\in I$ one has
\begin{equation}
 \| A^{-1} f - A_h^{-1}P_h f \|_{L^2(D)} 
 \leq c h^2 \| A^{-1} f \|_{W^{2,2}(D)}
\end{equation}
and
\begin{equation}
 \| A^{-1} f - A_h^{-1}P_h f \|_{L^{2}(D)}
 \leq 
 c h \| A^{-1} f \|_{W^{1,2}(D)}.
\end{equation}
This combined with~\eqref{fracPowEllip} and~\eqref{sobolevinterpol} results in estimate \eqref{app:ass} with $\delta\in \{\frac{1}{2},1\}$.
By interpolation (see \eqref{fracPowEllip}) we obtain, for any $\delta\in [\inv{2},1]$,\cgs{}
\begin{align*}
\n A^{-1} - A_h^{-1}P_h \n_{\calL(\calH^{A}_{\delta-1},\calH)} & \eqsim \n A^{-1} - A_h^{-1}P_h
\n_{\sgc{}\calL([\calH^A_{-1/2},\calH ]_{2-2\delta},\calH)\cgs{}}\\
& \lesssim h^{2\delta},
\end{align*}
with implied constant independent of $h$. Thus we may apply Theorem \ref{t:app} with $\delta=\eta$ to obtain the desired result.
\end{proof}

\begin{remark}\label{rem:not_all_eta}
Unfortunately it is not possible to obtain estimate \eqref{app:ass} in Theorem \ref{t:app} for $\delta=0$ in the setting of Proposition \ref{prop:finEl} as this would imply that the Riesz operator $R_h = A_h^{-1} P_h A$ is $L^2$-stable - which is not the case, see e.g.\ \cite[Section 1.5]{FuSaSu:01}. This restricts the applicability of Proposition \ref{prop:finEl}, as it only provides a convergence result if the the noise is sufficiently smooth, i.e., if $\theta_G\geq \inv{p}>0$.
\end{remark}
\section{Localization}\label{s:local}
The pathwise convergence results of Theorems \ref{c:galerkin_pathwise} and \ref{c:finEl_pathwise} remain valid
if $F$ and $G$ are merely locally Lipschitz and satisfy linear growth conditions. The argument by which this is demonstrated is entirely analogous to the argument presented in \cite{CoxNeer:11}, and we provide it here only for the reader's convenience.\par

Thus as before we consider the equation \eqref{SDE} under condition \MA{}, but instead of \MF{} and \MG{} we assume $F$ and $G$ to satisfy
\begin{itemize}
\item[\MFloc{}] For some $\theta_F>  -1$ the function
$F:[0,T]\times \calH\rightarrow \calH_{\theta_F}^{A}$ is locally Lipschitz continuous and uniformly of linear growth on $\calH$. That is to say, for every $m\in \N$ there exists a constant $C_{0,m}$ such that for all $t\in [0,T]$, and all $x_1,x_2,\in \calH$ such that $\n x_1\n_{\calH},\n x_2\n_{\calH} \leq m$ one has
\begin{align*}
\n F(t,x_1) - F(t,x_2) \n_{\calH_{\theta_F}^A} & \leq C_{0,m} \n x_1-x_2\n_{\calH}.
\end{align*}
Moreover, there exists a constant $C_1$ such that for all $t\in [0,T]$ and all $x\in \calH$ one has
\begin{align*}
\n F(t,x)\n_{\calH_{\theta_F}^A} &\leq
C_1(1+\n x\n_{\calH}).
\end{align*}\par
Finally, for all $x\in \calH$ we have that $F(\cdot,x):[0,T]\rightarrow \calH^{A}_{\theta_F}$ is measurable.
\item[\MGloc{}] For some  $\theta_G>-\inv{2}$ the function $G : [0,T]\times \calH\rightarrow \calL_2(H,\calH_{\theta_G}^{A})$ is locally Lipschitz continuous and uniformly of linear growth on
$\calH$.\par
Moreover, we have that the mapping $G(\cdot,x)h:[0,T]\rightarrow \calH_{\theta_G}^A$ is measurable for all $x\in \calH$ and all $h\in H$.
\end{itemize}

It has been proven in \cite{NVW08} that if one assumes \MFloc{} and \MGloc{}
instead of \MF{} and \MG{}, and moreover assumes that $x_0\in L^{0}(\Omega,\mathcal{F}_0;\calH)$, then equation
\eqref{SDE} has a unique mild solution in $L^0(C([0,T];\calH))$ for all $T>0$. The solution is constructed by approximations, which are obtained as follows.\par
For $m\in \N$ define
$F_{m}(t,x):=F(t,(1\minsym \frac{m}{\n x\n})x)$ and $G_m(t,x):=G\big(t,(1\minsym\frac{m}{ \n x \n })x\big)$.
Clearly $F_{m}$ and $G_{m}$ satisfy \MF{} and \MG{}. Suppose $\eta_0\geq 0$ is such that $x_0\in L^{0}(\Omega,\mathcal{F}_0;{\calH}_{\eta_0}^A)$. By aforementioned existence results (see page \pageref{varcons}) there exists, for all $p\in (2,\infty)$ satisfying $\inv{p}<\inv{2}+\theta_G$, a unique mild solution $U_{m}\in L^p(\Omega;C([0,T];\calH))$ to
\begin{equation}\label{SDEm}
\left\{ \begin{aligned} dU_{m}(t) & = AU_{m}(t)\,dt + F_{m}(t,U_{m}(t))\,dt \\
& \quad + G_{m}(t,U_{m}(t))\,dW_H(t);\quad t\in
[0,T],\\U_{m}(0)& = 1_{\{\n x_0\n_{{\calH}_{\eta_0}^A} \le m\}} x_0. \end{aligned}\right.
\end{equation}
Clearly we may take $\eta_0=0$ in the above if our aim is only to construct a solution. However, when proving convergence, it is essential to have $\eta_0>0$.\par
Fix $T>0$ and set
\begin{align*}
\tau_{m}^T(\omega) := \inf\{t\geq 0\, :\, \n U_{m}(t,\omega)\n_{\calH}\geq m\},
\end{align*}
with the convention that $\inf(\varnothing)=T$. By a
uniqueness argument one may show that for $m_1\leq m_2$ one has $U_{m_1}(t)=U_{m_2}(t)$
on $[0,\tau_{m_1}^T]$. Moreover, by \cite[Section 8]{NVW08} we have, due to the linear growth conditions on $F$ and $G$, that
$$\lim_{m\rightarrow \infty} \tau_{m}^T = T \quad\textrm{almost surely.}$$
In fact, because this holds for arbitrary $T>0$, there exists a set
$\Omega_0\subseteq \Omega$ of measure one\label{d:Omega0} such that for all $\omega\in\Omega_0$ there exists an
$m_{\omega}$ such that $\tau_{m}^T(\omega)=T$ for all $m\geq m_{\omega}$. \par

The mild solution $U$ to \eqref{SDE} with $F$ and $G$ satisfying \MFloc{} and \MGloc{} is
defined by setting
\begin{align*}
U(t,\omega) := \lim_{m\rightarrow \infty} U_{m}(t,\omega), \quad t\in [0,T], \ \omega\in \Omega_0,
\end{align*}
and $U(t,\omega):=0$ for $t\in [0,T]$ and $\omega \in \Omega\setminus \Omega_0$.
\par
Similarly, we may define $U^{n}$, the Galerkin approximation to $U$ in $\calH_n$ with $n\in\N$, by setting
\begin{align}\label{eq:defUnloc}
U^{(n)}(t,\omega) := \lim_{m\rightarrow \infty} U_{m}^{(n)}(t,\omega), \quad t\in [0,T], \ \omega\in \Omega_0^{(n)},
\end{align}
and $U^{(n)}(t,\omega):=0$ for $t\in [0,T]$ and $\omega \in \Omega\setminus \Omega_0^{(n)}$. Here $U^{(n)}_m$ denotes the process obtained by applying the Galerkin scheme in $\calH_n$, as considered in Section \ref{s:galerkin}, to equation \eqref{SDEm}, and $\Omega_{0}^{(n)}$ is defined to be the set on which the limit in equation \eqref{eq:defUnloc} exists. Note that by Theorem \ref{c:galerkin_pathwise} this set is of full measure.\par
\begin{corollary}[Localization of Theorem \ref{c:galerkin_pathwise}]\label{c:galerkin_local}
Suppose there exists an $\alpha>0$ and a constant $C$ such that for all $n\in \N$ we have
\begin{align*}
| \lambda_n | \geq C n^{\alpha}.
\end{align*}
Recall that $\eta_0>0$ is such that $x_0\in L^{0}(\Omega,\mathcal{F}_0,\calH_{\eta_0}^A)$ and let $\eta \in [0,\eta_0\minsym 1]$ satisfy
$$\eta<\min\{1+\theta_F,\tfrac{1}{2}+\theta_G\}.$$
Then there exists a random variable $\chi_{\eta}\in L^0(\Omega)$ such that
\begin{align*}
\n U - U^{(n)}\n_{C([0,T],\calH)}
&\leq \chi_{\eta} n^{-\alpha\eta}.
\end{align*}
\end{corollary}\par

Before providing a proof, let us state the analogous result for the finite element method of Section \ref{s:finEl}.\pars
Consider the setting of Section \ref{s:finEl}, \sgc{}in particular\cgs{}, we have $\calH=L^2(D)$ where $D$ is \sgc{}a convex polyhedron in $\R^d$, $d\in \N$, $A\colon W^{2,2}_B(D)\rightarrow L^2(D)$ is a self-adjoint second-order elliptic operator with homogeneous Dirichlet boundary conditions,
and $(V_h)_{h\in I}$ is a family of approximation spaces consisting of continuous piecewise polynomials of degree 1\cgs{}. Let $U$ be the solution provided by the approximations described above applied to equation \eqref{SDE} in this setting, with initial value $x_0=u_0\in L^{0}(\Omega,\mathcal{F}_0,W^{2\eta_0,2}_{B}(D))$ for some $\eta_0\geq 0$. \sgc{}For $h\in I$ one may define $U^{(h)}$, the finite element approximation to $U$ in $V_{h}$,\cgs{} by setting
\begin{align}\label{eq:defUnloc2}
U^{(\sgc{}h\cgs{})}(t,\omega) := \lim_{m\rightarrow \infty} U_{m}^{(\sgc{}h\cgs{})}(t,\omega), \quad t\in [0,T], \ \omega\in \Omega_0^{(\sgc{}h\cgs{})},
\end{align}
and $U^{(\sgc{}h\cgs{})}(t,\omega):=0$ for $t\in [0,T]$ and $\omega \in \Omega\setminus \Omega_0^{(\sgc{}h\cgs{})}$. Here $U_m^{(\sgc{}h\cgs{})}$ denotes the process obtained by applying the finite element scheme \eqref{finelsol}\sgc{}
with $F=F_m$ and $G=G_m$\cgs{}, and $\Omega_{0}^{(h)}$ is defined to be the set on which the limit in equation \eqref{eq:defUnloc2} exists. By Theorem \ref{c:finEl_pathwise} this set is of full measure.\par

\begin{corollary}[Localization of Theorem \ref{c:finEl_pathwise}]\label{c:finEl_local}
Let $x_0\in L^{0}(\Omega,\mathcal{F}_0,W^{2\eta_0,2}(D))$ for some $\eta_0\geq \inv{2}$. Let $\eta\in [\inv{2},\eta_0\minsym 1]$ be such that
$$\tinv{2} \leq \eta<\min\{1+\theta_F,\tfrac{1}{2}+\theta_G \}.$$ Then there exists a random variable $\chi_{\eta}\in L^0(\Omega)$ such
that for all \sgc{}$h\in I$\cgs{} we have
\begin{align*}
\n U - U^{(\sgc{}h\cgs{})}\n_{C([0,T],L^2(D))}
&\leq \chi_{\eta}(\omega) n^{-2\eta}.
\end{align*}
\end{corollary}\par

We only provide a proof for Corollary \ref{c:galerkin_local}; the proof of Corollary \ref{c:finEl_local} is entirely analogous.

\begin{proof}[Proof of Corollary \ref{c:galerkin_local}.]
Fix $\omega\in \bigcap_{n\in\N}\Omega_{0}^{(n)}\cap \Omega_0$. Let $m_{\omega}$ be such that $\tau_{m}^T(\omega)=T$ for all
$m\geq m_{\omega}$. Note that for all $m\geq m_{\omega}$ we have, by aforementioned uniqueness argument, $U_m(\omega)=U(\omega)$.\par
However, a priori this does not guarantee that $\n U^{(n)}_m(\omega)\n_{C([0,T];{\calH})}\leq m_{\omega}$ for $m\geq m_{\omega}$ and $n\in \N$; this requires an additional argument.\par

By Theorem \ref{c:galerkin_pathwise}, with $p\in (2,\infty)$ chosen such that 
$$\eta+\tinv{\alpha p}<\min\{1+\theta_F,\tfrac{1}{2}+\theta_G-\tinv{p}\},$$ there exists a constant $C_{\omega}$ depending on $\omega$ (and $m_{\omega}$), but independent of $n$, such that
\begin{align*}
\n U_{2m_{\omega}}(\omega) - U^{(n)}_{2m_{\omega}}(\omega) \n_{C([0,T];{\calH})}
& \leq C_{\omega} n^{-\alpha \eta}.
\end{align*}
In particular, for large enough $n$, say  $n\ge N_{\omega}$,
we have
\begin{align*}
\n U_{2m_{\omega}}(\omega) - U^{(n)}_{2m_{\omega}}(\omega) \n_{C([0,T];{\calH})}
& \leq m_{\omega}.
\end{align*}

As $\n U_{2m_{\omega}}(\omega)\n_{C([0,T];{\calH})}\leq m_{\omega}$, it follows that $\n U^{(n)}_{2m_{\omega}}(\omega)\n_{C([0,T];{\calH})}\leq 2m_{\omega}$ for $n\geq N_{\omega}$. Thus by definition of $F_{2m_{\omega}}$ and $G_{2m_{\omega}}$ and the uniqueness result of \cite[Lemma 7.2]{NVW08} we have, for $n\geq N_\omega$ and $t\in [0,T]$;
$$U^{(n)}_{2m_{\omega}}(\omega,t)=U^{(n)}(\omega,t).$$\par

Thus by Theorem \ref{c:galerkin_pathwise} applied to $U_{2m_{\omega}}$ and $U^{(n)}_{2m_{\omega}}$ it follows that there exists a constant $C_{\omega}$ depending on $\omega$, but independent of $n$, such that for $n\geq N_{\omega}$ one has
\begin{align*}
\n U(\omega) - U^{(n)}(\omega) \n_{C([0,T];\calH)}  &
= \n U_{2m_{\omega}}(\omega) - U^{(n)}_{2m_{\omega}}(\omega) \n_{C([0,T];{\calH})}  \leq C_{\omega} n^{-\alpha \eta}.
\end{align*}
It follows that there exists a $\tilde{C}_{\omega}$ such that for all $n\in \N$ one has
\begin{align*}
\n U(\omega) - U^{(n)}(\omega) \n_{C([0,T];\calH)} & \leq \tilde{C}_{\omega} n^{-\alpha \eta}.
\end{align*}
This proves Corollary \ref{c:galerkin_local}.\par
\end{proof}

\section*{Acknowledgment}
The authors wish to thank Jan van Neerven for helpful comments.

\end{document}